\documentclass[a4paper,12pt]{article}
\usepackage{amsthm}
\usepackage{amsmath,amssymb,latexsym,amsfonts,mathrsfs}

\theoremstyle{plain}
\newtheorem{Thm}{Theorem}[section]

\newtheorem{Lem}[Thm]{Lemma}
\newtheorem{Prop}[Thm]{Proposition}
\newtheorem{Cor}[Thm]{Corollary}

\theoremstyle{definition}

\newtheorem{Rem}[Thm]{Remark}
\newtheorem{Ex}[Thm]{Example}

\newcommand{\eps}{\ensuremath{\varepsilon}}
\newcommand{\e}{{\rm e}} 
\renewcommand{\d}{{\rm d}} %differential form
\newcommand{\tend}[2]{\mathrel{\mathop{\longrightarrow}\limits^{#1}_{#2}}}

\renewcommand{\hat}{\widehat}
\renewcommand{\tilde}{\widetilde}

\newcommand{\bN}{\ensuremath{\mathbb{N}}}

\newcommand{\bR}{\ensuremath{\mathbb{R}}}

\newcommand{\cB}{\ensuremath{\mathcal{B}}}
\newcommand{\cC}{\ensuremath{\mathcal{C}}}

\newcommand{\cE}{\ensuremath{\mathcal{E}}}
\newcommand{\cF}{\ensuremath{\mathcal{F}}}

\newcommand{\cK}{\ensuremath{\mathcal{K}}}

\newcommand{\cN}{\ensuremath{\mathcal{N}}}

\newcommand{\cS}{\ensuremath{\mathcal{S}}}

\newcommand{\sW}{\ensuremath{\mathscr{W}}}

\newcommand{\vn}{\ensuremath{\mbox{{\boldmath $n$}}}}

\numberwithin{equation}{section}

\makeatletter
\renewcommand\section{\@startsection {section}{1}{\z@}%
                                   {-3.5ex \@plus -1ex \@minus -.2ex}%
                                   {2.3ex \@plus.2ex}%
                                   {\normalfont\large\bf}}
\makeatother

\makeatletter
\renewcommand\subsection{\@startsection {subsection}{1}{\z@}%
                                   {-3.5ex \@plus -1ex \@minus -.2ex}%
                                   {2.3ex \@plus.2ex}%
                                   {\normalfont\normalsize\bf}}
\makeatother

\makeatletter
\@addtoreset{footnote}{page}
\makeatother

\newcommand{\absol}[1]{\left| #1 \right|} %absolute value
\newcommand{\rbra}[1]{\left( #1 \right)} %round brackets or parentheses
\newcommand{\cbra}[1]{\left\{ #1 \right\}} %curly brackets or braces
\newcommand{\sbra}[1]{\left[ #1 \right]} %square brackets or brackets
\newcommand{\abra}[1]{\left\langle #1 \right\rangle} %angle brackets or chevrons

\renewcommand{\subitem}{\par\hangindent 5mm \hspace*{-2mm}}
\renewcommand{\subsubitem}{\par\hangindent 10mm \hspace*{2mm}}

\begin{document}

\begin{center}
{\Large \bf Cameron--Martin formula for the $ \sigma $-finite measure 
unifying Brownian penalisations} 
\end{center}
\begin{center}
Kouji \textsc{Yano}\footnote{
Department of Mathematics, Graduate School of Science,
Kobe University, Kobe, JAPAN.}\footnote{
The research of this author was supported by KAKENHI (20740060).}
\end{center}
\begin{center}
{\small \today}
\end{center}
%\begin{center}
%{\small Dedicated to      on the occasion of his        birthday}
%\end{center}
%\bigskip

\begin{abstract}
Quasi-invariance under translation 
is established 
for the $ \sigma $-finite measure unifying Brownian penalisations, 
which has been introduced by Najnudel, Roynette and Yor 
%(\cite{NRY1}). 
({\em C. R. Math. Acad. Sci. Paris}, {\bf 345} no. 8, 459--466, 2007). 
For this purpose, 
the theory of Wiener integrals for centered Bessel processes, 
due to Funaki, Hariya and Yor 
%(\cite{FHY2}), 
({\em ALEA Lat. Am. J. Probab. Math. Stat.}, {\bf 1}, 225--240, 2006), 
plays a key role. 
\end{abstract}


%%%% text %%%%%

\section{Introduction}

Let $ \Omega = C([0,\infty ) \to \bR) $. 
Let $ (X_t:t \ge 0) $ denote the coordinate process 
and set $ \cF_{\infty } = \sigma(X_t:t \ge 0) $. 
We consider the following $ \sigma $-finite measure on $ (\Omega,\cF_{\infty }) $: 
\begin{align}
\sW = \int_0^{\infty } \frac{\d u}{\sqrt{2 \pi u}} \Pi^{(u)} \bullet R 
\label{eq: def of sW}
\end{align}
where $ \Pi^{(u)} \bullet R $ is given as follows: 
\subitem (i) 
$ \Pi^{(u)} $ denotes the law of the Brownian bridge from 0 to 0 of length $ u $; 
\subitem (ii) 
$ R $ denotes the law of the symmetrized 3-dimensional Bessel process; 
\subitem (iii) 
$ \Pi^{(u)} \bullet R $ denotes the concatenation of $ \Pi^{(u)} $ and $ R $. 

\noindent
This measure $ \sW $ 
has been introduced by Najnudel, Roynette and Yor (\cite{NRY1} and \cite{NRY2}) 
so that it unifies 
various Brownian penalisations. 
The Brownian penalisations can be explained roughly as follows 
(we will discuss details in Section \ref{sec: penal}): 
For a ``good" family $ \{ \Gamma_t(X) \} $ of non-negative $ \cF_{\infty } $-functionals 
such that $ \Gamma_t(X) \to \Gamma(X) $ as $ t \to \infty $, it holds that 
\begin{align}
\sqrt{\frac{\pi t}{2}} W[F_s(X)\Gamma_t(X)] 
\tend{}{t \to \infty } 
\sW[F_s(X)\Gamma(X)] 
\label{eq: rough penal}
\end{align}
for any bounded $ \cF_s $-measurable functional $ F_s(X) $. 

The purpose of this paper 
is to establish quasi-invariance of $ \sW $ 
under $ h $-translation when $ h $ belongs to {\em the Cameron--Martin type space}: 
\begin{align}
\cbra{ h \in \Omega : h_t = \int_0^t f(s) \d s 
\ \text{for some} \ 
f \in L^2(\d s) \cap L^1(\d s) } . 
\label{}
\end{align}
Now we state our main theorem. 

\begin{Thm} \label{thm: main}
Suppose that $ h_t = \int_0^t f(s) \d s $ with $ f \in L^2(\d s) \cap L^1(\d s) $. 
Then, for any non-negative $ \cF_{\infty } $-measurable functional $ F(X) $, 
it holds that 
\begin{align}
\sW[F(X+h)] = \sW[F(X) \cE(f;X)] 
\label{eq: main}
\end{align}
where 
\begin{align}
\cE(f;X) 
= \exp \rbra{ \int_0^{\infty } f(s) \d X_s 
- \frac{1}{2} \int_0^{\infty } f(s)^2 \d s } . 
\label{eq: intro cE}
\end{align}
\end{Thm}

Theorem \ref{thm: main} will be proved in Section \ref{sec: proof}. 

Theorem \ref{thm: main} involves {\em Wiener integral}, i.e., 
the stochastic integral $ \int_0^{\infty } f(s) \d X_s $ 
of a deterministic function $ f $. 
(To avoid confusion, we give the following remark: 
In \cite{CM1} and \cite{CM2}, the Wiener integral means 
the integral with respect to the Wiener measure.) 
The author has proved in his recent work \cite{Y1} that 
this Wiener integral is well-defined 
if $ f \in L^2(\d s) \cap L^1(\frac{\d s}{1 + \sqrt{s}}) $, 
i.e., 
\begin{align}
\int_0^{\infty } |f(s)|^2 \d s + \int_0^{\infty } |f(s)| \frac{\d s}{1 + \sqrt{s}} < \infty . 
\label{}
\end{align}
Note the obvious inclusion: $ L^1(\d s) \subset L^1(\frac{\d s}{1 + \sqrt{s}}) $. 
We will discuss details in Section \ref{sec: Wint}. 
One may conjecture that Theorem \ref{thm: main} is valid 
for $ h_t = \int_0^t f(s) \d s $ with $ f \in L^2(\d s) \cap L^1(\frac{\d s}{1 + \sqrt{s}}) $, 
but we have not succeeded at this point.

We give several remarks which help us to understand Theorem \ref{thm: main} deeply. 

$ \mathbf{1^{\circ}).} $ 
{\bf Rephrasing the main theorem.} 
Let $ g(X) $ denote the last exit time from 0 for $ X $: 
\begin{align}
g(X) = \sup \{ u \ge 0 : X_u = 0 \} . 
\label{}
\end{align}
For $ u \ge 0 $, let $ \theta_u X $ denote the shifted process: 
$ (\theta_u X)_s = X_{u+s} $, $ s \ge 0 $. 
Then the definition \eqref{eq: def of sW} says that 
the measure $ \sW $ can be described as follows: 
\subitem {\rm (i)} 
$ \displaystyle \sW(g(X) \in \d u) = \frac{\d u}{\sqrt{2 \pi u}} $; 
\subitem {\rm (ii)} 
For (Lebesgue) a.e. $ u \in [0,\infty ) $, 
it holds that, given $ g(X)=u $, 
\subsubitem {\rm (iia)} 
$ (X_s:s \le u) $ is a Brownian bridge from 0 to 0 of length $ u $; 
\subsubitem {\rm (iib)} 
$ ((\theta_u X)_s:s \ge 0) $ is 
a symmetrized 3-dimensional Bessel process. 

\noindent
In the same manner as this, 
Theorem \ref{thm: main} can be rephrased as the following corollary. 
We write $ T_h^* \sW $ for the image measure of $ X+h $ under $ \sW $. 
For $ u \in [0,\infty ) $, we define 
\begin{align}
\cE_u(f;X) 
= \exp \rbra{ \int_0^u f(s) \d X_s 
- \frac{1}{2} \int_0^u f(s)^2 \d s } . 
\label{}
\end{align}

\begin{Cor}
Suppose that $ h_t = \int_0^t f(s) \d s $ with $ f \in L^2(\d s) \cap L^1(\d s) $. 
Then it holds that 
\begin{align}
T_h^* \sW = \int_0^{\infty } \d u \, \rho^f(u) 
\, \Pi^{(u),f} \bullet R^{f(\cdot+u)} 
\label{}
\end{align}
where 
\begin{align}
\rho^f(u) 
=& \frac{1}{\sqrt{2 \pi u}} \Pi^{(u)} \sbra{\cE_u(f;\cdot)} R \sbra{ \cE(f(\cdot+u);\cdot) } 
, 
\label{} \\
\Pi^{(u),f}(\d X) 
=& \frac{\cE_u(f;X) \Pi^{(u)}(\d X)}{\Pi^{(u)} \sbra{\cE_u(f;\cdot)}} 
, 
\label{} \\
R^{f(\cdot+u)}(\d X) 
=& \frac{\cE(f(\cdot+u);X) R(\d X)}{R \sbra{\cE(f(\cdot+u);\cdot)}} . 
\label{}
\end{align}
In other words, the law of the process $ X+h $ under $ \sW $ may be described as follows: 
\subitem {\rm (i)$ ' $} 
$ \displaystyle \sW(g(X+h) \in \d u) = \rho^f(u) \d u $; 
\subitem {\rm (ii$ ' $)} 
For a.e. $ u \in [0,\infty ) $, 
it holds that, given $ g(X+h)=u $, 
\subsubitem {\rm (iia$ ' $)} 
$ (X_s+h_s:s \le u) $ has law $ \Pi^{(u),f} $; 
\subsubitem {\rm (iib$ ' $)} 
$ ((\theta_u(X+h))_s:s \ge 0) $ has law $ R^{f(\cdot+u)} $. 
\end{Cor}

$ \mathbf{2^{\circ}).} $ 
{\bf Sketch of the proof.}
We will divide the proof of Theorem \ref{thm: main} into the following steps: 
\subitem 
Step 1. 
$ \sW[F(X+h_{\cdot \wedge T})] = \sW[F(X) \cE_T(f;X)] $ 
for $ 0<T<\infty $; 
\subitem 
Step 2. 
$ \sW[F(X) \cE_T(f;X)] \to \sW[F(X) \cE(f;X)] $ as $ T \to \infty $; 
\subitem 
Step 3. 
$ \sW[F(X+h_{\cdot \wedge T})] \to \sW[F(X+h)] $ as $ T \to \infty $. 

\noindent
Note that, in Steps 2 and 3, 
we will confine ourselves to certain particular classes of test functions $ F $. 

One may think that 
Step 1 should be immediate from the following rough argument using \eqref{eq: rough penal}: 
For any ``good" $ \cF_s $-measurable functional $ F_s(X) $, 
\begin{align}
\sW[F_s(X+h_{\cdot \wedge T}) \Gamma(X+h_{\cdot \wedge T})] 
=& \lim_{t \to \infty } \sqrt{\frac{\pi t}{2}} 
W[F_s(X+h_{\cdot \wedge T}) \Gamma_t(X+h_{\cdot \wedge T})] 
\label{} \\
=& \lim_{t \to \infty } \sqrt{\frac{\pi t}{2}} 
W[F_s(X) \cE_T(f;X) \Gamma_t(X)] 
\label{} \\
=& \sW[F_s(X) \cE_T(f;X) \Gamma(X)] . 
\label{}
\end{align}
This observation, however, should be justified carefully, 
because the functional $ \cE_T(f;X) $ is not bounded. 
We shall utilize {\em Markov property} for $ \{ (X_t),\sW \} $ 
(see Subsection \ref{sec: Markov} for the details): 
\begin{align}
\sW[ F_T(X) G(\theta_T X) ] = W \sbra{ F_T(X) \sW_{X_T}[G(\cdot)] } 
\label{eq: exit law}
\end{align}
where $ \sW_x $ is the image measure of $ x+X $ under $ \sW(\d X) $. 
The identity \eqref{eq: exit law} suggests, in a way, that 
$ \{ \sW_x : x \in \bR \} $ is a family of {\em exit laws} 
whose transition up to finite time is the Brownian motion, 
while the Markov property of the Brownian motion asserts that 
\begin{align}
W[ F_T(X) G(\theta_T X) ] = W \sbra{ F_T(X) W_{X_T}[G(\cdot)] } . 
\label{}
\end{align}
This makes a remarkable contrast with {\em It\^o's excursion law} $ \vn $ (see \cite{Ito}), 
which satisfies the Markov property: 
\begin{align}
\vn[ F_T(X) G(\theta_T X) ] = \vn \sbra{ F_T(X) W^0_{X_T}[G(\cdot)] } 
\label{}
\end{align}
where $ \{ (X_t),(W_x^0) \} $ denotes the Brownian motion 
killed upon hitting the origin. 
In other words, $ \vn $ produces a family of {\em entrance laws} 
whose transition after positive time is the killed Brownian motion.

We remark again that the Wiener integral $ \int_0^{\infty } f(s) \d X_s $ is not Gaussian. 
In order to prove necessary estimates involving Wiener integrals in Step 2, 
we utilize the theory of Wiener integrals for {\em centered Bessel processes}, 
which is due to Funaki, Hariya and Yor \cite{FHY2}. 
For the 3-dimensional Bessel process $ \{ (X_t),R^+_a \} $ starting from $ a \ge 0 $, we define 
\begin{align}
\hat{X}^{(a)}_t = X_t - R^+_a[X_t] 
\label{}
\end{align}
and call $ \{ (\hat{X}^{(a)}_t),R^+_a \} $ {\em the centered Bessel process}. 
We shall apply, to the convex function $ \psi(x)=(\e^{|x|}-1)^2 $, 
the following theorem, 
which was proved by Funaki--Hariya--Yor \cite{FHY2} via {\em Brascamp--Lieb inequality} \cite{BL}, 
and from which we derive our necessary estimates. 

\begin{Thm}[\cite{FHY2}] \label{thm: FHY}
For any $ f \in L^2(\d s) $ and 
any non-negative convex function $ \psi $ on $ \bR $, it holds that 
\begin{align}
R^+_a \sbra{ \psi \rbra{ \int_0^{\infty } f(t) \d \hat{X}^{(a)}_t } } 
\le W \sbra{ \psi \rbra{ \int_0^{\infty } f(t) \d X_t } } . 
\label{}
\end{align}
\end{Thm}

For the proof of Theorem \ref{thm: FHY}, see \cite[Prop.4.1]{FHY2}.

$ \mathbf{3^{\circ}).} $ 
{\bf Comparison with the Brownian case.} 
Let us recall the well-known {\em Cameron--Martin formula} for Brownian motion 
(see \cite{CM1} and \cite{CM2}). 
Let $ W $ stand for the Wiener measure on $ \Omega $ with $ W(X_0=0)=1 $. 

It is well-known that, if $ h_t = \int_0^t f(s) \d s $ with $ f \in L^2(\d s) $, 
\begin{align}
W [F(X+h)] = W [ F(X) \cE(f;X) ] 
\label{eq: CM}
\end{align}
for any non-negative $ \cF_{\infty } $-measurable functional $ F(X) $. 
It is also well-known that, if $ h \notin H $, 
the image measure of $ X+h $ under $ W(\d X) $ 
is mutually singular on $ \cF_{\infty } $ to $ W(\d X) $. 

It is immediate from \eqref{eq: CM} that, 
if $ h_t = \int_0^t f(s) \d s $ with $ f \in L^2_{\rm loc}(\d s) $, then 
\begin{align}
W [F_t(X+h)] = W [ F_t(X) \cE_t(f;X) ] 
\label{eq: qi local}
\end{align}
for any non-negative $ \cF_t $-measurable functional $ F_t(X) $ 
where 
\begin{align}
\cE_t(f;X) 
= \exp \rbra{ \int_0^t f(s) \d X_s 
- \frac{1}{2} \int_0^t f(s)^2 \d s } . 
\label{}
\end{align}

Now we give some remarks about comparison 
between the two cases of $ W $ and $ \sW $. 

{\rm (i)} 
Let $ f \in L^2(\d s) $. As a corollary of \eqref{eq: CM}, we see that 
$ W[\cE(f;X)] < \infty $ and, consequently, that 
$ W[\cE(f;X)^p] < \infty $ for any $ p \ge 1 $. 
This shows that, 
if $ F(X) \in L^p(W(\d X)) $ for some $ p>1 $, then $ F(X+h) \in L^1(W(\d X)) $. 

Let $ f \in L^2(\d s) \cap L^1(\d s) $. 
In the case of $ \sW $, however, 
we see immediately by taking $ F \equiv 1 $ in \eqref{eq: main} that 
\begin{align}
\sW[\cE(f;X)] = \infty , 
\label{}
\end{align}
which we should always keep in mind. 
Now the following question arises: 
\begin{align}
\sW[\cE(f;X) \Gamma(X)] < \infty 
\label{eq: finiteness}
\end{align}
holds for what functional $ \Gamma(X) $? 
The problem is that we do not know the distribution 
of the Wiener integral $ \int_0^{\infty } f(s) \d X_s $ under $ \sW $; 
in fact, it is no longer Gaussian! 
In Theorem \ref{thm: nondeg}, 
we will appeal to a certain penalisation result 
and establish \eqref{eq: finiteness} 
for {\em Feynman--Kac functionals} $ \Gamma(X) $, 
the class of which we shall introduce in Subsection \ref{sec: FK penal}. 

{\rm (ii)} 
In the Brownian case, we have the following criterion: 
The $ h $-translation of $ W $ is quasi-invariant or singular with respecto to $ W $ 
according as $ h \in L^2(\d s) $ or $ h \notin L^2(\d s) $, respectively. 

In the case of $ \sW $, however, 
we do not know what happens on $ \sW(\d X) $ 
when $ h \notin H $ or when $ f \notin L^1(\d s) $. 

{\rm (iii)} 
Let $ f \in L^2_{\rm loc}(\d s) $. 
In the Brownian case, we have the quasi-invariance \eqref{eq: qi local} on each $ \cF_t $. 
In the case of $ \sW(\d X) $, however, 
we find a drastically different situation (see Theorem \ref{thm: NRY2}): 
For any non-negative $ \cF_t $-measurable functional $ F_t(X) $, 
\begin{align}
\sW[F_t(X+h)] = \sW[F_t(X)] = 0 \ \text{or} \ \infty 
\label{}
\end{align}
according as $ W(F_t(X)=0)=1 $ or $ W(F_t(X)=0)<1 $.

$ \mathbf{4^{\circ}).} $ 
{\bf Integration by parts formulae.} 
From the Cameron--Martin theorem \eqref{eq: CM} in the Brownian case, 
we immediately obtain the following integration by parts formula: 
\begin{align}
W[\nabla_h F(X)] = W \sbra{ F(X) \int_0^{\infty } f(s) \d X_s } 
\label{}
\end{align}
for $ h_t = \int_0^t f(s) \d s $ with $ f \in L^2(\d s) $ 
and for any good functional $ F(X) $, 
where $ \nabla $ denotes the {\em Gross--Sobolev--Malliavin derivative} 
(see, e.g., \cite{Ust}). 
In the case of $ \sW $, from Theorem \ref{thm: main}, 
we may expect the following integration by parts formula: 
\begin{align}
\sW[\partial_h F(X)] = \sW \sbra{ F(X) \int_0^{\infty } f(s) \d X_s } 
\label{eq: ibyp}
\end{align}
for $ h_t = \int_0^t f(s) \d s $ with $ f \in L^2(\d s) \cap L^1(\d s) $ 
and for any good functional $ F(X) $, 
where $ \partial_h $ is in the G\^ateaux sense. 
We have not succeeded in finding a reasonable class of functionals $ F $ 
for which both sides of \eqref{eq: ibyp} make sense and coincide. 

Let us give a remark 
about 3-dimensional Bessel bridge of length $ u $ from 0 to 0, 
which we denote by $ \{ (X_s:s \in [0,u]) , R^{+,(u)} \} $. 
Although we do not have the Cameron--Martin formula for the bridge, 
there is a remarkable result due to Zambotti (\cite{Z02} and \cite{Z03}) 
that the following integration by parts formula holds: 
\begin{align}
R^{+,(1)}[\partial_h F(X)] 
= R^{+,(1)} \sbra{ F(X) \int_0^{\infty } f(s) \d X_s } 
+ {\rm (BC)} 
\label{}
\end{align}
for $ h_t = \int_0^t f(s) \d s $ with $ f $ satisfying a certain regularity condition 
and for any good functional $ F(X) $, 
where $ \partial_h $ is in the G\^ateaux sense 
and where 
\begin{align}
{\rm (BC)} = - \int_0^1 \frac{\d u h_u}{\sqrt{2 \pi u^3 (1-u)^3}} 
\rbra{ R^{+,(u)} \bullet R^{+,(1-u)} } [F(\cdot)] . 
\label{}
\end{align}
The remainder term {\rm (BC)} 
may describe the {\em boundary contribution}. 
Indeed, the measure $ R^{+,(1)} $ is supported 
on the set of non-negative continuous paths on $ [0,1] $, 
while the measure $ R^{+,(u)} \bullet R^{+,(1-u)} $ is supported 
on the subset of paths which hit 0 once and only once; 
the latter set may be regarded in a certain sense 
as the boundary of the former. 
See also Bonaccorsi--Zambotti \cite{Z04}, Zambotti \cite{Z05}, 
Hariya \cite{Hariya} and Funaki--Ishitani \cite{FI} 
for similar results about integration by parts formulae.

The organization 
of this paper is as follows. 
In Section \ref{sec: penal}, 
we recall several results of Brownian penalisations. 
In Section \ref{sec: Wint}, 
we study Wiener integrals for the processes considered. 
Section \ref{sec: proof} is devoted to the proofs 
of our main theorems.

\section{Brownian penalisations} \label{sec: penal}

\subsection{Notations}

Let $ X=(X_t:t \ge 0) $ denote the coordinate process 
of the space $ \Omega = C([0,\infty );\bR) $ of continuous functions 
from $ [0,\infty ) $ to $ \bR $. 
Let $ \cF_t = \sigma(X_s: s \le t) $ for $ 0<t<\infty $ 
and $ \cF_{\infty } = \sigma(\cup_t \cF_t) $. 
For $ 0<u<\infty $, 
we write $ X^{(u)}=(X_t: 0 \le t \le u) $ 
and $ \Omega^{(u)} = C([0,u];\bR) $.

\noindent
$ \mathbf{1^{\circ}).} $ 
{\bf Brownian motion}. 
For $ a \in \bR $, we denote by $ W_a $ the Wiener measure on $ \Omega $ 
with $ W_a(X_0=a)=1 $. 
We simply write $ W $ for $ W_0 $. 

\noindent
$ \mathbf{2^{\circ}).} $ 
{\bf Brownian bridge}. 
We denote by $ \Pi^{(u)} $ the law on $ \Omega^{(u)} $ of the {\em Brownian bridge}: 
\begin{align}
\Pi^{(u)}(\cdot) = W(\cdot|X_u=0) . 
\label{}
\end{align}
The process $ X^{(u)} $ under $ \Pi^{(u)} $ 
is a centered Gaussian process with covariance 
$ \Pi^{(u)}[X_s X_t] = s-st/u $ for $ 0 \le s \le t \le u $. 
As a realization of $ \{ X^{(u)},\Pi^{(u)} \} $, we may take 
\begin{align}
\cbra{ B_s-\frac{s}{u}B_u : s \in [0,u] } . 
\label{}
\end{align}

\noindent
$ \mathbf{3^{\circ}).} $ 
{\bf 3-dimensional Bessel process}. 
For $ a \ge 0 $, 
we denote by $ R_a^+ $ the law on $ \Omega $ 
of the 3-dimensional Bessel process starting from $ a $, 
i.e., the law of the process $ (\sqrt{Z_t}) $ 
where $ (Z_t) $ is the unique strong solution 
to the stochastic differential equation 
\begin{align}
\d Z_t = 2 \sqrt{|Z_t|} \d \beta_t + 3 \d t 
, \quad Z_0 = a^2 
\label{}
\end{align}
with $ (\beta_t) $ a one-dimensional standard Brownian motion. 
Under $ R_a^+ $, the process $ X $ satisfies 
\begin{align}
\d X_t = \d B_t + \frac{1}{X_t} \d t 
, \quad X_0 = a 
\label{eq: 3BES SDE}
\end{align}
with $ \{ (B_t),R_a^+ \} $ a one-dimensional standard Brownian motion. 

For $ a>0 $, we denote by $ R^-_{-a} $ 
the law on $ \Omega $ of $ (-X_t) $ under $ R^+_a $. 
We define 
\begin{align}
R_a = 
\begin{cases}
R^+_a \quad \text{if} \ a>0 , \\
R^-_a \quad \text{if} \ a<0 
\end{cases}
\label{}
\end{align}
and 
\begin{align}
R = R_0 = \frac{R_0^+ + R_0^-}{2} ; 
\label{}
\end{align}
in other words, 
$ R $ is the law on $ \Omega $ of $ (\eps X_t) $ 
under the product measure $ P(\d \eps) \otimes R_0^+(\d X) $ 
where $ P(\eps=1)=P(\eps=-1)=1/2 $. 

\noindent
$ \mathbf{4^{\circ}).} $ 
{\bf The $ \sigma $-finite measure $ \sW $}. 
For $ u>0 $ and for two processes 
$ X^{(u)} = (X_t:0 \le t \le u) $ and $ Y = (Y_t: t \ge 0) $, 
we define the concatenation $ X^{(u)} \bullet Y $ 
as 
\begin{align}
(X^{(u)} \bullet Y)_t = 
\begin{cases}
X_t     \quad & \text{if} \ 0 \le t < u , \\
Y_{t-u} \quad & \text{if} \ t \ge u \ \text{and} \ X_u=Y_0 , \\
X_u       \quad & \text{if} \ t \ge u \ \text{and} \ X_u \neq Y_0 . 
\end{cases}
\label{}
\end{align}
We define the concatenation $ \Pi^{(u)} \bullet R $ 
as the law of $ X^{(u)} \bullet Y $ 
under the product measure $ \Pi^{(u)}(\d X^{(u)}) \otimes R(\d Y) $. 
Then we define 
\begin{align}
\sW = \int_0^{\infty } \frac{\d u}{\sqrt{2 \pi u}} \Pi^{(u)} \bullet R . 
\label{eq: defsW}
\end{align}
For $ x \in \bR $, we define $ \sW_x $ 
as the image measure of $ x+X $ under $ \sW(\d X) $; 
in other words, 
\begin{align}
\sW_x[F(X)] = \sW[F(x+X)] 
\label{eq:defsWx}
\end{align}
for any non-negative $ \cF_{\infty } $-measurable functional $ F(X) $.

\noindent
$ \mathbf{5^{\circ}).} $ 
{\bf Random times}. 
For $ a \in \bR $, we denote {\em the first hitting time of $ a $} by 
\begin{align}
\tau_a(X) = \inf \{ t>0 : X_t=a \} . 
\label{}
\end{align}
We denote {\em the last exit time from 0} by 
\begin{align}
g(X) = 
\sup \{ t \ge 0 : X_t = 0 \} . 
\label{}
\end{align}

\subsection{Feynman--Kac penalisations} \label{sec: FK penal}

Let $ L_t^y(X) $ denote the local time by time $ t $ of level $ y $: 
For $ W_x(\d X) $-a.e. $ X $, it holds that 
\begin{align}
\int_0^t 1_A(X_s) \d s = \int_A L_t^y(X) \d y 
, \quad A \in \cB(\bR) , \ t \ge 0 . 
\label{}
\end{align}
For a non-negative Borel measure $ V $ on $ \bR $ and a process $ (X_t) $ under $ W(\d X) $, 
we write 
\begin{align}
\cK_t(V;X) = \exp \rbra{ - \int_{\bR} L_t^x(X) V(\d x) } 
\label{eq: cKt}
\end{align}
and 
\begin{align}
\cK(V;X) = \exp \rbra{ - \int_{\bR} L_{\infty }^x(X) V(\d x) } . 
\label{eq: cK}
\end{align}
The following theorem is due 
to Roynette--Vallois--Yor \cite{RVY1}, \cite[Thm.4.1]{RVY0} and \cite[Thm.2.1]{RY}. 

\begin{Thm}[{\cite[Thm.2.1]{RY}}] \label{thm: penal}
Let $ V $ be a non-negative Borel measure on $ \bR $ 
and suppose that 
\begin{align}
0 < \int_{\bR} (1+|x|) V(\d x) < \infty . 
\label{eq: V ibility}
\end{align}
Then the following statements hold: 
\subitem {\rm (i)} 
$ \displaystyle \varphi_V(x) := \lim_{t \to \infty } \sqrt{\frac{\pi t}{2}} W_x[\cK_t(V;X)] $ 
and the limit exists in $ \bR_+ $; 
\subitem {\rm (ii)} 
$ \varphi_V $ is the unique solution 
of the Sturm--Liouville equation 
\begin{align}
\varphi_V''(x) = 2 \varphi_V(x) V(\d x) 
\label{eq: SLeq1}
\end{align}
in the sense of distributions (see, e.g., \cite[Appendix \S 8]{RevuzYor}) 
subject to the boundary condition: 
\begin{align}
\lim_{x \to - \infty } \varphi_V'(x) = -1 
\quad \text{and} \quad 
\lim_{x \to \infty } \varphi_V'(x) = 1 ; 
\label{eq: SLeq2}
\end{align}
\subitem {\rm (iii)} 
For any $ 0<s<\infty $ and any bounded $ \cF_s $-measurable functional $ F_s(X) $, 
\begin{align}
\frac{W_x[F_s(X) \cK_t(V;X)]}{W_x[\cK_t(V;X)]} 
\to W_x \sbra{ F_s(X) \frac{\varphi_V(X_s)}{\varphi_V(X_0)} \cK_s(V;X) } 
\quad \text{as} \ t \to \infty ; 
\label{}
\end{align}
\subitem {\rm (iv)} 
$ \displaystyle \rbra{ M_s^{(V)}(X) 
:= \frac{\varphi_V(X_s)}{\varphi_V(X_0)} \cK_s(V;X) : s \ge 0 } $ 
is a $ (W_x,(\cF_s)) $-martingale which converges a.s. to 0 as $ s \to \infty $; 
\subitem {\rm (v)} 
Under the probability measure $ W_x^{(V)} $ on $ \cF_{\infty } $ induced by the relation 
\begin{align}
W_x^{(V)}[F_s(X)] = W_x[F_s(X) M_s^{(V)}(X)] , 
\label{eq: equiv between WxV and Wx}
\end{align}
the process $ (X_t) $ solves the stochastic differential equation 
\begin{align}
X_t = x + B_t + \int_0^t \frac{\varphi_V'}{\varphi_V}(X_s) \d s 
\label{}
\end{align}
where $ (B_t) $ is a $ (W_x^{(V)},(\cF_t)) $-Brownian motion starting from 0; 
in particular, the process $ (X_t) $ is a transient diffusion 
which admits the following function $ \gamma_V(x) $ as its scale function: 
\begin{align}
\gamma_V(x) = \int_0^x \frac{\d y}{\varphi_V(y)^2} . 
\label{}
\end{align}
\end{Thm}

\begin{Rem} \label{rem: linear bound}
By (ii) of Theorem \ref{thm: penal}, we see that 
the function $ \varphi_V $ also enjoys the following properties: 
\subitem {\rm (vi)} 
$ \varphi_V(x) \sim |x| $ as $ x \to \infty $. 
This suggests that the process $ \{ (X_t),(W_x^{(V)}) \} $ 
behaves like 3-dimensional Bessel process when the value of $ |X_t| $ is large. 
\subitem {\rm (vii)} 
$ \inf_{x \in \bR} \varphi_V(x) > 0 $. 
This shows that the origin is regular for itself. 
\end{Rem}

\begin{Ex}[A key example for \cite{RVY1}] \label{ex: penal}
Suppose that $ V = \lambda \delta_0 $ with some $ \lambda > 0 $ 
where $ \delta_0 $ denotes the Dirac measure at $ 0 $. 
That is, 
\begin{align}
\cK_t(\lambda \delta_0;X) = \exp \rbra{ - \lambda L_t^0(X) } . 
\label{}
\end{align}
Then we can solve equation \eqref{eq: SLeq1}-\eqref{eq: SLeq2} 
and consequently we obtain 
\begin{align}
\varphi_{\lambda \delta_0}(x) = \frac{1}{\lambda} + |x| , 
\label{eq: varphi1}
\end{align}
\begin{align}
M_t^{(\lambda \delta_0)}(X) = \rbra{ 1 + \lambda |X_t| } \exp \rbra{ - \lambda L_t^0(X) } 
\label{eq: Mslambdadelta0}
\end{align}
and 
\begin{align}
X_t = x + B_t + \int_0^t \frac{{\rm sgn}(X_s)}{\frac{1}{\lambda} + |X_s|} \d s 
\quad \text{under $ W_x^{(V)} $} . 
\label{}
\end{align}
\end{Ex}

\subsection{The universal $ \sigma $-finite measure}

Najnudel--Roynette--Yor (\cite{NRY1} and \cite{NRY2}) 
introduced the measure $ \sW $ on $ \cF_{\infty } $ defined by \eqref{eq: defsW} 
to give a global view on the Brownian penalisations. 
It unifies the Feynman--Kac penalisations 
in the sense of the following theorem, 
which is due 
to Najnudel--Roynette--Yor \cite[Thm.1.1.2 and Thm.1.1.6]{NRY2}; 
see also Yano--Yano--Yor \cite[Thm.8.1]{YYY}. 
See also Najnudel--Nikeghbali (\cite{NN1} and \cite{NN2}) 
for careful treatment of augmentation of filtrations.

\begin{Thm}[\cite{NRY2}] \label{thm: NRY1}
Let $ x \in \bR $ 
and let $ V $ be a non-negative measure on $ \bR $ satisfying \eqref{eq: V ibility}. 
Then 
it holds that 
\begin{align}
\sW_x \sbra{ Z_t(X) \cK(V;X) } = W_x \sbra{ Z_t(X) \varphi_V(X_t) \cK_t(V;X) } 
\label{eq: local equiv sW and W}
\end{align}
for any $ t \ge 0 $ 
and any non-negative $ \cF_t $-measurable functional $ Z_t(X) $, 
where $ \cK(V;X) $ has been defined as \eqref{eq: cK}. 
Consequently, it holds that 
\begin{align}
\varphi_V(x) = \sW_x \sbra{ \cK(V;X) } 
\label{eq: varphi2}
\end{align}
and that 
\begin{align}
W_x^{(V)}(\d X) = \frac{1}{\varphi_V(x)} \cK(V;X) \sW_x(\d X) 
\quad \text{on $ \cF_{\infty } $}. 
\label{eq: equiv between WxV and sW}
\end{align}
\end{Thm}

The following theorem can be found in \cite[p.6, Point v) and Thm.1.1.6]{NRY2}; 
see also \cite[Thm.5.1]{YYY}. 

\begin{Thm}[\cite{NRY2}] \label{thm: NRY2}
The following statements hold: 
\subitem {\rm (i)} 
$ \displaystyle \sW(g(X) \in \d u) = \frac{\d u}{\sqrt{2 \pi u}} $ 
on $ [0,\infty ) $. 
\\
In particular, 
$ \sW $ is $ \sigma $-finite on $ \cF_{\infty } $; 
\subitem {\rm (ii)} 
For $ A \in \cF_t $ with $ 0<t<\infty $, 
$ \displaystyle \sW(A) = 
\begin{cases}
0 \ & \text{if} \ W(A)=0 , \\
\infty \ & \text{if} \ W(A)>0 . 
\end{cases}
$ 
\\
In particular, $ \sW $ is {\em not} $ \sigma $-finite on $ \cF_t $. 
\end{Thm}

We give the proof for completeness of this paper. 

\begin{proof}
Claim (i) is obvious by definition \eqref{eq: def of sW} of $ \sW $. 
Let us prove Claim (ii). 
Let $ 0<t<\infty $. 
Suppose that $ A \in \cF_t $ and $ W(A)=0 $. 
Then we have $ \sW[1_A \cK(\delta_0;X)]=0 $ by \eqref{eq: local equiv sW and W}, 
which implies that $ \sW(A)=0 $. 
Suppose in turn that $ A \in \cF_t $ and $ W(A)>0 $. 
For $ \lambda>0 $, we apply \eqref{eq: local equiv sW and W} for $ V= \lambda \delta_0 $ 
and we have 
\begin{align}
\sW(A) 
\ge 
\sW \sbra{ 1_A \e^{ - \lambda L^0_{\infty } } } 
= W \sbra{ 1_A \rbra{ \frac{1}{\lambda} + |X_t| } \e^{ - \lambda L^0_t } } 
\ge \frac{1}{\lambda} W \sbra{ 1_A \e^{ - \lambda L^0_t } } . 
\label{}
\end{align}
Letting $ \lambda \to 0+ $, we obtain, by the monotone convergence theorem, that 
$ W[1_A \e^{ - \lambda L^0_t }] \to W(A) > 0 $, 
and consequently, that $ \sW(A) = \infty $. 
\end{proof}

We also need the following property. 

\begin{Prop} \label{thm: tau0}
For $ x \in \bR $, it holds that 
\begin{align}
\sW_x(\tau_0(X)=\infty ) = |x| . 
\label{}
\end{align}
\end{Prop}

\begin{proof}
By symmetry, we have only to prove the claim for $ x \ge 0 $. 
Let $ V=\delta_0 $ and $ F(X) = 1_{\{ \tau_0(X)=\infty \}} $. 
Note that $ L_{\infty }^0(X) = 0 $ if $ \tau_0(X)=\infty $. 
Hence it follows from Example \ref{ex: penal} and Theorem \ref{thm: NRY1} that 
\begin{align}
\sW_x(\tau_0(X)=\infty ) = \varphi_{\delta_0}(x) W_x^{(\delta_0)}(\tau_0(X)=\infty ) . 
\label{}
\end{align}
Since $ \varphi_{\delta_0}(x) = 1 + x $ 
and since $ \gamma_{\delta_0}(x) = \frac{x}{1+x} $, we have 
\begin{align}
\sW_x(\tau_0(X)=\infty ) 
= (1 + x) \cdot 
\frac{\gamma_{\delta_0}(x)-\gamma_{\delta_0}(0)}{\gamma_{\delta_0}(\infty )-\gamma_{\delta_0}(0)} 
= x . 
\label{}
\end{align}
The proof is complete. 
\end{proof}

\subsection{Markov property of $ \{ (X_t),(\cF_t),(\sW_x) \} $} \label{sec: Markov}

We may say that 
$ \{ (X_t),(\cF_t),(\sW_x) \} $ possesses 
Markov property in the following sense. 

\begin{Thm}[\cite{NRY1} and \cite{NRY2}] \label{thm: Markov}
Let $ x \in \bR $ and $ t \ge 0 $. 
Let $ F $ be a non-negative $ \cF_{\infty } $-measurable functional. 
Then it holds that 
\begin{align}
\sW_x[ Z_t(X) F(\theta_t X) ] = W_x \sbra{ Z_t(X) \sW_{X_t}[F(\cdot)] } 
\label{eq: Mp}
\end{align}
for any non-negative $ \cF_t $-measurable functional $ Z_t(X) $. 
Moreover, the constant time $ t $ in \eqref{eq: Mp} 
may be replaced by any finite $ (\cF_t) $-stopping time $ \tau $. 
\end{Thm}

\begin{proof}
Let $ V $ be as in Theorem \ref{thm: penal}. 
Then we have 
\begin{align}
& \sW_x[Z_t(X) \cK_t(V;X) F(\theta_t X) \cK(V;\theta_t X)] 
\label{} \\
=& \sW_x[Z_t(X) F(\theta_t X) \cK(V;X)] 
\quad \text{(by the multiplicativity property of $ \cK(V;X) $)}
\label{} \\
=& \varphi_V(x) W_x^{(V)} \sbra{ Z_t(X) F(\theta_t X) } 
\quad \text{(by \eqref{eq: equiv between WxV and sW})}
\label{} \\
=& \varphi_V(x) W_x^{(V)} \sbra{ Z_t(X) W_{X_t}^{(V)}[F(\cdot)] } 
\quad \text{(by the Markov property of $ W_{\cdot}^{(V)} $)}
\label{} \\
=& \varphi_V(x) W_x \sbra{ Z_t(X) W_{X_t}^{(V)}[F(\cdot)] 
\cdot \frac{\varphi_V(X_t)}{\varphi_V(X_0)} \cK_t(V;X) } 
\quad \text{(by \eqref{eq: equiv between WxV and Wx})}
\label{} \\
=& W_x \sbra{ Z_t(X) \cK_t(V;X) 
\sW_{X_t} \sbra{ F(\cdot) \cK(V;\cdot) } } 
\quad \text{(by \eqref{eq: equiv between WxV and sW}).}
\label{}
\end{align}
Taking $ V=\lambda \delta_0 $ and letting $ \lambda \to 0+ $, 
we obtain \eqref{eq: Mp} by the monotone convergence theorem. 
In the same way, we can prove \eqref{eq: Mp} 
also in the case where the constant time $ t $ is replaced by a finite stopping time $ \tau $. 
\end{proof}

Since the measure $ \sW_x $ has infinite total mass, 
we cannot consider conditional expectation in the usual sense. 
But, by the help of Theorem \ref{thm: Markov}, 
we can introduce a counterpart in the following sense. 

\begin{Cor}[\cite{NRY1} and \cite{NRY2}; see also \cite{YYY}]
Let $ x \in \bR $ and $ t \ge 0 $. 
Let $ F $ be a $ \cF_{\infty } $-measurable functional 
which is in $ L^1(\sW_x) $. 
Then there exists a unique $ \{ (\cF_t),W_x \} $-martingale 
$ M_t[F;X] $ such that 
\begin{align}
\sW_x[ Z_t(X) F(X) ] = W_x[Z_t(X) M_t[F;X]] 
\label{}
\end{align}
for any bounded $ \cF_t $-measurable functional $ Z_t(X) $. 
Moreover, it is given as 
\begin{align}
M_t[F;X] = \int_{\Omega} \sW_{X_t}(\d Y) F(X^{(t)} \bullet Y) 
\quad \text{$ W_x(\d X) $-a.s.} 
\label{eq: Mp1}
\end{align}
\end{Cor}

\begin{Rem}
If $ F \in L^1(W_x) $, then the family of the conditional expectations 
$ \{ W_x[F|\cF_t] : t \ge 0 \} $ is a uniformly integrable martingale. 
In contrast with this fact, if $ F \in L^1(\sW_x) $, 
the martingale $ \{ M_t[F;X] : t \ge 0 \} $ under $ W_x $ converges to 0 as $ t \to \infty $, 
and consequently, it is not uniformly integrable. 
\end{Rem}

\begin{Rem}
Since $ M_t $ is an operator from $ L^1(\sW_x) $ to $ L^1(W_x) $, 
we do not have a counterpart of the tower property for the usual conditional expectation. 
\end{Rem}

\begin{Ex}
Let $ V $ be a non-negative measure on $ \bR $ satisfying \eqref{eq: V ibility}. 
Then (iv) and (v) of Theorem \ref{thm: penal} may be rewritten as 
\begin{align}
M_t \sbra{ \cK(V;\cdot) ; X } 
= \varphi_V(X_t) \cK_t(V;X) . 
\label{}
\end{align}
From this and from Remark \ref{rem: linear bound}, we see that 
\begin{align}
M_t \sbra{ \cK(V;\cdot) ; X } 
\in L^p(W) 
\quad \text{for any $ p \ge 1 $}. 
\label{eq: LpW}
\end{align}
In particular, formula \eqref{eq: Mslambdadelta0} may be rewritten as 
\begin{align}
M_t \sbra{ \cK(\lambda \delta_0;\cdot) ; X } 
= \rbra{ \frac{1}{\lambda} + |X_t| } \cK_t(\lambda \delta_0;X) . 
\label{eq: sWdelta0}
\end{align}
\end{Ex}

\section{Wiener integrals} \label{sec: Wint}

Let $ \cS $ denote 
the set of all step functions $ f $ on $ [0,\infty ) $ of the form: 
\begin{align}
f(t) = \sum_{k=1}^n c_k 1_{[t_{k-1},t_k)}(t) 
, \quad t \ge 0 
\label{eq: step func}
\end{align}
with $ n \in \bN $, 
$ c_k \in \bR $ ($ k = 1,\ldots,n $) 
and $ 0=t_0 < t_1 < \cdots < t_n < \infty $. 
Note that $ \cS $ is dense in $ L^2(\d s) $. 
For a function $ f \in \cS $ and a process $ X $, 
we define 
\begin{align}
\int_0^{\infty } f(t) \d X_t = \sum_{k=1}^n c_k (X_{t_k} - X_{t_{k-1}}) . 
\label{eq: pre Wiener integ}
\end{align}
If $ \int_0^{\infty } f(t) \d X_t $ can be defined 
as the limit in some sense 
of $ \int_0^{\infty } f_n(t) \d X_t $ 
for an approximating sequence $ \{ f_n \} $ of $ f $, 
then we will call it {\em Wiener integral} of $ f $ for the process $ X $. 

We have the following facts: 
If a sequence $ \{ f_n \} \subset \cS $ 
approximates $ f $ in $ L^2(\d s) $, then it holds that 
\begin{align}
\int_0^{\infty } f_n(s) \d X_s \tend{}{n \to \infty } \int_0^{\infty } f(s) \d X_s 
\quad \text{in $ W $-probability} 
\label{}
\end{align}
and that, for any $ u>0 $, 
\begin{align}
\int_0^u f_n(s) \d X_s \tend{}{n \to \infty } \int_0^u f(s) \d X_s 
\quad \text{in $ \Pi^{(u)} $-probability} . 
\label{}
\end{align}

\subsection{Wiener integral for 3-dimensional Bessel process}

Let $ p_t(x) $ denote the density of the Brownian semigroup: 
\begin{align}
p_t(x) = \frac{1}{\sqrt{2 \pi t}} \exp \rbra{ - \frac{x^2}{2t} } 
, \quad t>0 , \ x \in \bR . 
\label{}
\end{align}
Let $ a \ge 0 $ be fixed. 
It is well-known (see, e.g., \cite[\S VI.3]{RevuzYor}) that, for $ t>0 $ and $ x>0 $, 
\begin{align}
R^+_a (X_t \in \d x) = 
\begin{cases}
\displaystyle 
\frac{x}{a} \cbra{ p_t(x-a) - p_t(x+a) } \d x , 
& a>0 , \\
\displaystyle 
\frac{2 x^2}{t} p_t(x) \d x , 
& a=0 . 
\end{cases}
\label{}
\end{align}
From this formula, it is straightforward that, for $ t>0 $ and $ x>0 $, 
\begin{align}
\phi_a(t) := R^+_a \sbra{ \frac{1}{X_t} } = 
\begin{cases}
\displaystyle 
\frac{1}{a} \int_{-a}^a p_t(x) \d x 
, & a > 0 \\
\displaystyle 
2 p_t(0) = \sqrt{\frac{2}{\pi t}} 
, & a = 0 . 
\end{cases}
\label{eq: phia}
\end{align}
Since $ p_t(x) \le p_t(0) $, it is obvious by definition that 
\begin{align}
\phi_a(t) \le \phi_0(t) 
, \quad a>0 , \ t > 0 . 
\label{eq: monotonicity}
\end{align}
Note that $ \phi_a(t) $ has the following asymptotics as $ t \to 0+ $: 
\begin{align}
\phi_a(t) \sim 
\begin{cases}
1/a               & \text{if $ a>0 $}, \\
\sqrt{2/(\pi t)} & \text{if $ a=0 $}. 
\end{cases}
\label{}
\end{align}
By the stochastic differential equation \eqref{eq: 3BES SDE}, we see that 
\begin{align}
R^+_a[X_t] = a + \int_0^t R^+_a \sbra{ \frac{1}{X_s} } \d s 
= a + \int_0^t \phi_a(s) \d s , 
\label{}
\end{align}

Now the following lemma is obvious. 

\begin{Lem}
Let $ f \in L^2(\d s) \cap L^1(\phi_a(s) \d s) $. 
Then, according to the stochastic differential equation \eqref{eq: 3BES SDE}, 
the Wiener integral may be defined as 
\begin{align}
\int_0^{\infty } f(s) \d X_s 
= \int_0^{\infty } f(s) \d B_s + \int_0^{\infty } \frac{f(s)}{X_s} \d s . 
\label{}
\end{align}
If a sequence $ \{ f_n \} \subset \cS $ 
approximates $ f $ both in $ L^2(\d s) $ and in $ L^1(\phi_a(s) \d s) $, i.e., 
\begin{align}
\int_0^{\infty } |f_n(s)-f(s)|^2 \d s  
+ \int_0^{\infty } |f_n(s)-f(s)| \phi_a(s) \d s 
\tend{}{n \to \infty } 
0, 
\label{eq: approx seq}
\end{align}
then it holds that 
\begin{align}
\int_0^{\infty } f_n(s) \d X_s \tend{}{n \to \infty } \int_0^{\infty } f(s) \d X_s 
\quad \text{in $ R^+_a $-probability} . 
\label{}
\end{align}
\end{Lem}

Following Funaki--Hariya--Yor (\cite{FHY2}), 
we may propose another way of constructing the Wiener integral. 
We define 
\begin{align}
\hat{X}^{(a)}_s = X_s - R^+_a[X_s] 
\label{}
\end{align}
and we call $ \{ (\hat{X}^{(a)}_s),R^+_a \} $ the {\em centered Bessel process}. 
We simply write $ \hat{X}_s $ for $ \hat{X}^{(0)}_s $. 
By applying Theorem \ref{thm: FHY} with $ \psi(x)=x^2 $, we obtain the following fact: 
If a sequence $ \{ f_n \} \subset \cS $ 
approximates $ f $ in $ L^2(\d s) $, then it holds that 
\begin{align}
\int_0^{\infty } f_n(s) \d \hat{X}^{(a)}_s 
\tend{}{n \to \infty } \int_0^{\infty } f(s) \d \hat{X}^{(a)}_s 
\quad \text{in $ R^+_a $-probability} . 
\label{}
\end{align}
We then obtain the following lemma. 

\begin{Lem}
Let $ f \in L^2(\d s) \cap L^1(\phi_a(s) \d s) $. 
Then it holds that 
\begin{align}
\int_0^{\infty } f(s) \d X_s 
= \int_0^{\infty } f(s) \d \hat{X}^{(a)}_s + \int_0^{\infty } f(s) \phi_a(s) \d s 
\quad \text{$ R^+_a $-a.s.}
\label{}
\end{align}
\end{Lem}

\subsection{Wiener integral for $ X $ under $ \sW $}

Define 
\begin{align}
L^1_+(\sW) = \cbra{ G:\Omega \to \bR_+, \ \text{$ \cF $-measurable}, \ 
\sW(G = 0)=0 , \ \sW[G]<\infty } . 
\label{}
\end{align}
For $ G \in L^1_+(\sW) $, 
we define a probability measure $ \sW^G $ on $ (\Omega,\cF) $ by 
\begin{align}
\sW^G(A) = \frac{\sW[1_A G]}{\sW[G]} 
, \quad A \in \cF . 
\label{eq: sWG}
\end{align}
We recall the following notion of convergence. 

\begin{Prop} \label{prop: loc in meas}
Let $ Z,Z_1,Z_2,\ldots $ be $ \cF_{\infty } $-measurable functionals. 
Then the following statements are equivalent: 
\subitem {\rm (i)} 
For any $ \eps>0 $ and any $ A \in \cF $ with $ \sW(A)<\infty $, 
it holds that \\
$ \sW \rbra{ A \cap \cbra{ |Z_n-Z| \ge \eps } } \to 0 $. 
\subitem {\rm (ii)} 
$ Z_n \to Z $ in $ \sW^G $-probability for some $ G \in L^1_+(\sW) $. 
\subitem {\rm (iii)} 
$ Z_n \to Z $ in $ \sW^G $-probability for any $ G \in L^1_+(\sW) $. 
\subitem {\rm (iv)} 
One can extract, from an arbitrary subsequence, 
a further subsequence $ \{ n(k):k=1,2,\ldots \} $ along which 
$ Z_{n(k)} \to Z $ $ \sW $-a.e. 

\noindent
If one (and hence all) of the above statements holds, 
then we say that 
\begin{align}
\text{$ Z_n \to Z $ {\em locally in $ \sW $-measure}.} 
\label{}
\end{align}
\end{Prop}

For the proof of Proposition \ref{prop: loc in meas}, see, e.g., \cite{Y1}.

Wiener integral for $ X $ under $ \sW(\d X) $ 
may be defined with the help of the following theorem. 

\begin{Thm}[\cite{Y1}] \label{thm: approx}
Let $ f \in L^2(\d s) \cap L^1(\frac{\d s}{1 + \sqrt{s}}) $. 
Suppose that a sequence $ \{ f_n \} \subset \cS $ approximates $ f $ 
both in $ L^2(\d s) $ and in $ L^1(\frac{\d s}{1 + \sqrt{s}}) $, i.e., 
\begin{align}
\int_0^{\infty } |f_n(s)-f(s)|^2 \d s 
+ \int_0^{\infty } |f_n(s)-f(s)| \frac{\d s}{1 + \sqrt{s}} 
\tend{}{n \to \infty } 
0. 
\label{}
\end{align}
(Note that this condition is strictly weaker than the condition \eqref{eq: approx seq}.) 
Then it holds that 
\begin{align}
\int_0^{\infty } f_n(s) \d X_s 
\tend{}{n \to \infty } 
\int_0^{\infty } f(s) \d X_s 
\quad \text{locally in $ \sW $-measure} . 
\label{}
\end{align}
Moreover, there exists a functional $ J(f;u,X) $ 
measurable with respect to the product $ \sigma $-field 
$ \cB([0,\infty )) \otimes \cF_{\infty } $ such that 
\begin{align}
\int_0^{\infty } f(s) \d X_s 
= J(f;g(X),X) 
\quad \text{$ \sW $-a.e.}
\label{}
\end{align}
and that it holds $ \d u $-a.e. that 
\begin{align}
J(f;u,X^{(u)} \bullet Y) = \int_0^u f(s) \d X_s + \int_0^{\infty } f(s+u) \d Y_s 
\label{}
\end{align}
is valid a.e. with respect to $ \Pi^{(u)}(\d X^{(u)}) \otimes R(\d Y) $. 
\end{Thm}

The following lemma allows us to use the same notation for Wiener integrals 
under $ W(\d X) $ and $ \sW(\d X) $. 
Let us temporarily write $ I^W(f;X) $ (resp. $ I^{\sW}(f;X) $) 
for the Wiener integral $ I(f;X) $ under $ W(\d X) $ (resp. $ \sW(\d X) $). 

\begin{Lem} \label{lem: equivalence}
Suppose that there exist $ F \in L^1(\sW) $ and $ G \in L^1(W) $ such that 
\begin{align}
\sW \sbra{ H(X) F(X) } = W \sbra{ H(X) G(X) } 
\label{}
\end{align}
holds for any bounded measurable functional $ H(X) $. 
Then, for any $ f \in L^2(\d s) \cap L^1(\frac{\d s}{1 + \sqrt{s}}) $, it holds that 
\begin{align}
\sW \sbra{ \varphi(I^{\sW}(f;X)) H(X) F(X) } 
= W \sbra{ \varphi(I^{W}(f;X)) H(X) G(X) } 
\label{eq: equiv lem}
\end{align}
for any bounded Borel function $ \varphi $ on $ \bR $. 
\end{Lem}

\begin{proof}
This is obvious by Theorem \ref{thm: approx} and by the dominated convergence theorem. 
\end{proof}

\subsection{Integrability lemma}

For later use, we need the following lemma. 

\begin{Lem} \label{lem L1}
Let $ f \in L^1(\d s) $. Define 
\begin{align}
\tilde{f}(t) 
= \int_0^{\infty } |f(s+t)| \frac{\d s}{\sqrt{s}} 
= \int_t^{\infty } |f(s)| \frac{\d s}{\sqrt{s-t}} 
, \quad t > 0 . 
\label{eq: def of g}
\end{align}
Then the following statements hold: 
\subitem {\rm (i)} 
For any $ a>0 $, it holds that 
\begin{align}
\int_0^a \tilde{f}(t) \d t 
\le 2 \sqrt{a} \int_0^{\infty } |f(s)| \d s ; 
\label{}
\end{align}
\subitem {\rm (ii)} 
There exists a sequence $ t(n) \to \infty $ such that 
$ \tilde{f}(t(n)) \to 0 $. 
\end{Lem}

\begin{proof}
(i) Let $ a > 0 $. Then we have 
\begin{align}
\int_0^a \tilde{f}(t) \d t 
=& \int_0^a \d t \int_t^a |f(s)| \frac{\d s}{\sqrt{s-t}} 
+ \int_0^a \d t \int_a^{\infty } |f(s)| \frac{\d s}{\sqrt{s-t}} 
\label{} \\
=& \int_0^a \d s |f(s)| \int_0^s \frac{\d t}{\sqrt{s-t}} 
+ \int_a^{\infty } \d s |f(s)| \int_0^a \frac{\d t}{\sqrt{s-t}} 
\label{} \\
\le& \int_0^a |f(s)| (2 \sqrt{s}) \d s 
+ \int_a^{\infty } \d s |f(s)| \int_0^a \frac{\d t}{\sqrt{a-t}} 
\label{} \\
\le& 2 \sqrt{a} \int_0^{\infty } |f(s)| \d s . 
\label{}
\end{align}

(ii) Let $ 0<a<b<\infty $. Then we have 
\begin{align}
\frac{(b-a)}{\sqrt{b}} \inf_{t:t>a} \tilde{f}(t) 
\le \frac{1}{\sqrt{b}} \int_a^b \tilde{f}(t) \d t 
\le 2 \int_0^{\infty } |f(s)| \d s . 
\label{}
\end{align}
Since $ (b-a)/\sqrt{b} \to \infty $ as $ b \to \infty $ with $ a $ fixed, 
we wee that $ \inf_{t:t>a} \tilde{f}(t) = 0 $ for any $ a>0 $. This implies that 
\begin{align}
\liminf_{t \to \infty } \tilde{f}(t) = 0 . 
\label{}
\end{align}
The proof is now complete. 
\end{proof}

\section{Cameron--Martin formula} \label{sec: proof}

For a function $ h_t = \int_0^t f(s) \d s $ 
with $ f \in L^2(\d s) \cap L^1(\frac{\d s}{1 + \sqrt{s}}) $ 
and a process $ (X_s) $ under $ \sW_x $ for $ x \in \bR $, 
we write 
\begin{align}
\cE_t(f;X) = \exp \rbra{ \int_0^t f(s) \d X_s - \frac{1}{2} \int_0^t f(s)^2 \d s } 
\label{}
\end{align}
and 
\begin{align}
\cE(f;X) 
= \exp \rbra{ \int_0^{\infty } f(s) \d X_s 
- \frac{1}{2} \int_0^{\infty } f(s)^2 \d s } . 
\label{}
\end{align}
In what follows, 
let $ V $ be a non-negative Borel measure satisfying \eqref{eq: V ibility}.

\subsection{The first step}

\begin{Prop} \label{thm: main lem}
Let $ h_t = \int_0^t f(s) \d s $ with $ f \in L^2(\d s) $ and let $ T>0 $. 
Then, 
for any non-negative $ \cF_{\infty } $-measurable functional $ F(X) $, 
it holds that 
\begin{align}
\sW[F(X+h_{\cdot \wedge T})] = \sW[F(X) \cE_{T}(f;X)] . 
\label{eq: main lem}
\end{align}
If, moreover, $ M_T[F;X] \in L^p(W) $ for some $ p>1 $, 
then $ F(X+h_{\cdot \wedge T}) \in L^1(\sW) $. 
\end{Prop}

\begin{proof}
Let $ t \ge T $ be fixed. 
By the multiplicativity property of $ \cK(\delta_0;\cdot) $ 
and since $ h_{(\cdot + t) \wedge T} = h_T $, we have 
\begin{align}
\cK(\delta_0;X + h_{\cdot \wedge T}) 
= \cK_t(\delta_0;X + h_{\cdot \wedge T}) 
\cK(\delta_0;\theta_t X + h_T) . 
\label{}
\end{align}
Let $ G_t(X) $ be a non-negative $ \cF_t $-measurable functional. 
Then, by the Markov property \eqref{eq: Mp1}, 
by \eqref{eq:defsWx} and by \eqref{eq: varphi1},  
we have 
\begin{align}
M_t \sbra{ \cK(\delta_0;\cdot + h_{\cdot \wedge T}) ; X } 
=& \cK_t(\delta_0;X + h_{\cdot \wedge T}) 
\sW_{X_t} \sbra{ \cK(\delta_0;X + h_T) } 
\label{} \\
=& \cK_t(\delta_0;X + h_{\cdot \wedge T}) 
\sW_{X_t+h_T} \sbra{ \cK(\delta_0;X) } 
\label{} \\
=& \cK_t(\delta_0;X + h_{\cdot \wedge T}) 
(1+|X_t+h_T|) . 
\label{}
\end{align}
Hence we obtain 
\begin{align}
\begin{split}
& \sW \sbra{ G_t(X+h_{\cdot \wedge T}) \cK(\delta_0;X + h_{\cdot \wedge T}) } 
\\
=& W \sbra{ G_t(X+h_{\cdot \wedge T}) \cK_t(\delta_0;X + h_{\cdot \wedge T}) 
\rbra{ 1+|X_t+h_T| } } . 
\end{split}
\label{eq: main lem pf1}
\end{align}
By the Cameron--Martin formula \eqref{eq: CM}, 
by formula \eqref{eq: sWdelta0}, 
and then by the Markov property \eqref{eq: Mp}, we have 
\begin{align}
\text{\eqref{eq: main lem pf1}} 
=& W \sbra{ G_t(X) \cK_t(\delta_0;X) 
\rbra{ 1 + |X_t| } \cE_T(f;X) } 
\label{} \\
=& W \sbra{ G_t(X) M_t[\cK(\delta_0;\cdot);X] \cE_T(f;X) } 
\label{} \\
=& \sW \sbra{ G_t(X) \cK(\delta_0;X) \cE_T(f;X) } . 
\label{}
\end{align}
Since $ t \ge T $ is arbitrary, we see that 
\begin{align}
\sW \sbra{ G(X+h_{\cdot \wedge T}) \cK(\delta_0;X + h_{\cdot \wedge T}) } 
= 
\sW \sbra{ G(X) \cK(\delta_0;X) \cE_T(f;X) } 
\label{}
\end{align}
holds for any non-negative $ \cF_{\infty } $-measurable functional $ G(X) $. 
Replacing the functional $ G(X) $ by $ F(X) \cK(\delta_0;X)^{-1} $, 
we obtain \eqref{eq: main lem}. 

Suppose that $ M_T[F;X] \in L^p(W) $ for some $ p>1 $. 
Since $ \cE_T(h;X) $ is $ \cF_T $-measurable, we have 
\begin{align}
\sW[F(X) \cE_T(f;X)] 
=& W[M_T[F;X] \cE_T(f;X)] 
\label{} \\
\le& W[M_T[F;X]^p]^{1/p} W[\cE_T(f;X)^q]^{1/q} 
< \infty 
\label{}
\end{align}
where $ q $ is the conjugate exponent to $ p $: $ (1/p) + (1/q) = 1 $. 
The proof is now complete. 
\end{proof}

\subsection{Integrability under $ \sW $, when weighed by Feynman--Kac functionals}

We need the following theorem. 

\begin{Thm} \label{thm: nondeg}
Let $ h_t = \int_0^t f(s) \d s $ with $ f \in L^2(\d s) \cap L^1(\d s) $. 
Let $ V $ be as in Theorem \ref{thm: penal} 
and set $ C_V = \inf_{x \in \bR} \varphi_V(x) > 0 $. 
Then it holds that 
\begin{align}
\sW \sbra{ \cK(V;X) \cE(f;X) } 
\le& \varphi_V(0) \exp \rbra{ \frac{1}{C_V} \| f \|_{L^1(\d s)} } . 
\label{}
\end{align}
\end{Thm}

\begin{proof}
By Theorem \ref{thm: NRY1}, we have 
\begin{align}
\frac{1}{\varphi_V(0)} 
\sW \sbra{ \cK(V;X) \cE(f;X) } = W^{(V)} \sbra{ \cE(f;X) } . 
\label{eq: nondeg1}
\end{align}
By (v) of Theorem \ref{thm: penal}, we see that 
\begin{align}
\text{\eqref{eq: nondeg1}} = 
W^{(V)} \sbra{ \cE(f;B) 
\exp \rbra{ \int_0^{\infty } f(s) \frac{\varphi_V'}{\varphi_V}(X_s) \d s } 
} 
\label{eq: nondeg2}
\end{align}
where $ \{ (B_t),W^{(V)} \} $ is a Brownian motion. 
Since $ |\varphi_V'(x)| \le 1 $ and $ \varphi_V(x) \ge C_V $ for any $ x \in \bR $, we have 
\begin{align}
\text{\eqref{eq: nondeg2}} \le 
W^{(V)} \sbra{ \cE(f;B) } 
\exp \rbra{ \frac{1}{C_V} \int_0^{\infty } |f(s)| \d s } . 
\label{}
\end{align}
Since $ W^{(V)}[\cE(f;B)] = 1 $, we obtain the desired inequality. 
\end{proof}

\subsection{The second step}

We utilize the following lemma. 

\begin{Lem} \label{lem2}
Let $ h_t = \int_0^t f(s) \d s $ with $ f \in L^2(\d s) \cap L^1(\d s) $. 
Then, for any $ 0<s<\infty $, it holds that 
\begin{align}
\sW \sbra{ \cE_t(f;X) \e^{-g(X)} ; g(X)>t } 
\tend{}{t \to \infty } 
0 . 
\label{eq: gX>t conv to zero}
\end{align}
\end{Lem}

\begin{proof}
By the Markov property \eqref{eq: Mp}, we see that 
\begin{align}
\sW \sbra{ \cE_t(f;X) \e^{-g(X)} ; g(X)>t } 
= W \sbra{ \cE_t(f;X) \e^{-t} 
\sW_{X_t} \sbra{ \e^{-g(X)} ; \tau_0(X) < \infty } } . 
\label{eq: gX>t conv to zero prf}
\end{align}
By the strong Markov property \eqref{eq: Mp}, we see, for any $ x \in \bR $, that 
\begin{align}
\sW_x \sbra{ \e^{-g(X)} ; \tau_0(X) < \infty } 
= W_x \sbra{ \e^{-\tau_0(X)} } \sW_0 \sbra{ \e^{-g(X)} } 
\le \int_0^{\infty } \frac{\d u}{\sqrt{2 \pi u}} \e^{-u} = \frac{1}{\sqrt{2}}. 
\label{}
\end{align}
Hence we obtain 
\begin{align}
\text{\eqref{eq: gX>t conv to zero prf}} 
\le \frac{1}{\sqrt{2}} \e^{-t} W \sbra{ \cE_t(f;X) } 
= \frac{1}{\sqrt{2}} \e^{-t} 
\tend{}{t \to \infty } 0 . 
\label{}
\end{align}
The proof is now complete. 
\end{proof}

\begin{Lem} \label{lem2-2}
Let $ h_t = \int_0^t f(s) \d s $ with $ f \in L^2(\d s) \cap L^1(\d s) $. 
Let $ V $ be as in Theorem \ref{thm: penal}. 
Then it holds that 
\begin{align}
\sW \sbra{ \cE(f;X) \cK(V;X) \e^{-g(X)} ; g(X)>t } 
\tend{}{t \to \infty } 
0 . 
\label{eq: gX>t conv to zero2}
\end{align}
\end{Lem}

\begin{proof}
Since $ \sW[\cE(f;X) \cK(V;X)] < \infty $ by Theorem \ref{thm: nondeg}. 
The desired conclusion is now obvious by the dominated convergence theorem. 
\end{proof}

\begin{Lem} \label{lem3-2}
Let $ h_t = \int_0^t f(s) \d s $ with $ f \in L^2(\d s) \cap L^1(\d s) $. 
Set 
\begin{align}
\tilde{f}(t) = \int_0^{\infty } |f(s+t)| \frac{\d s}{\sqrt{s}} 
, \quad t>0 , 
\label{}
\end{align}
\begin{align}
\sigma_t = \| f(\cdot+t) \| 
= \cbra{ \int_t^{\infty } f(s)^2 \d s }^{1/2} 
, \quad t>0 , 
\label{}
\end{align}
and set 
\begin{align}
E(t) = 
E \sbra{ \absol{ \exp \cbra{ \sigma_t |\cN| 
+ c \tilde{f}(t) + \frac{1}{2} \sigma_t^2 } - 1 }^2 } 
, \quad t>0 
\label{}
\end{align}
where $ \cN $ stands for the standard Gaussian variable 
and $ c = \sqrt{2/\pi} $. 
Then it holds that 
\begin{align}
R_a \sbra{ \absol{ \cE(f(\cdot + t);\cdot) - 1 }^2 } \le E(t) 
\quad \text{for any $ t>0 $ and any $ a \in \bR $}. 
\label{}
\end{align}
\end{Lem}

\begin{proof}
Let us write $ \abra{f,g} = \int_0^{\infty } f_1(s) f_2(s) \d s $ for $ f_1,f_2 \in L^2(\d s) $. 
Note that 
\begin{align}
\cE(f(\cdot + t);X) 
= \exp \cbra{ \int_0^{\infty } f(s+t) \d \hat{X}^{(a)}_s 
+ \abra{ f(\cdot+t),\phi_a } 
- \frac{1}{2} \sigma_t^2 } 
\quad \text{under $ R_a^+ $} . 
\label{}
\end{align}
Since $ |e^b-1| \le e^{|b|}-1 $ for any $ b \in \bR $, we have 
\begin{align}
\absol{ \cE(f(\cdot + t);\cdot) - 1 }^2 
\le 
\absol{ \exp \cbra{ \absol{ \int_0^{\infty } f(s+t) \d \hat{X}^{(a)}_s 
+ \abra{ f(\cdot+t),\phi_a } 
- \frac{1}{2} \sigma_t^2 } } 
-1 }^2 . 
\label{}
\end{align}
Since, for any constant $ b \in \bR $, 
$ \psi(x) = (\e^{|x+b|}-1)^2 $ is a convex function, 
we may apply Theorem \ref{thm: FHY} and obtain 
\begin{align}
R_a^+ \sbra{ \absol{ \cE(f(\cdot + t);\cdot) - 1 }^2 } 
\le E \sbra{ \absol{ \exp \cbra{ 
\absol{ \sigma_t \cN + \abra{ f(\cdot + t),\phi_a } - \frac{1}{2} \sigma_t^2 } } - 1 
}^2 } . 
\label{}
\end{align}
Since 
\begin{align}
\absol{ \abra{ f(\cdot + t),\phi_a } } 
\le \abra{ |f(\cdot + t)|,\phi_0 } 
= c \tilde{f}(t) , 
\label{}
\end{align}
we obtain the desired result. 
\end{proof}

\begin{Lem} \label{lem3}
Let $ h_t = \int_0^t f(s) \d s $ with $ f \in L^2(\d s) \cap L^1(\d s) $. 
Then there exists a sequence $ t(n) \to \infty $ such that 
\begin{align}
\sW \sbra{ \e^{-g(X)} \cK(V;X) \absol{ \cE(f;X)-\cE_{t(n)}(f;X) } } 
\to 0 . 
\label{}
\end{align}
\end{Lem}

\begin{proof}
By Lemmas \ref{lem2} and \ref{lem2-2}, it suffices to prove that 
\begin{align}
\sW \sbra{ \e^{-g(X)} \cK(V;X) \absol{ \cE(f;X)-\cE_t(f;X) } ; g(X) \le t } 
\label{eq: lem3-1}
\end{align}
converges to 0 along some sequence $ t=t(n) \to \infty $. 

By the multiplicativity: 
\begin{align}
\cE(f;X) = \cE_t(f;X) \cE(f(\cdot + t);\theta_t X) , 
\label{}
\end{align}
we have 
\begin{align}
\text{\eqref{eq: lem3-1}} 
=& \sW \sbra{ \e^{-g(X)} \cK(V;X) \cE_t(f;X) \absol{ \cE(f(\cdot + t);\theta_t X) -1 } ; 
g(X) \le t } . 
\label{eq: lem3-2}
\end{align}
By the Schwarz inequality, \eqref{eq: lem3-2} is dominated by $ A^{1/2} B^{1/2} $ where 
\begin{align}
A = \sW \sbra{ \cK(V;X)^2 \cE_t(f;X)^2 } 
\label{}
\end{align}
and 
\begin{align}
B = \sW \sbra{ \e^{-2g(X)} \absol{ \cE(f(\cdot + t);\theta_t X) -1 }^2 ; 
g(X) \le t } . 
\label{}
\end{align}
By Theorem \ref{thm: nondeg}, we see that 
\begin{align}
A 
\le& \sW \sbra{ \cK(2V;X) \cE(2f1_{[0,t)};X) } \exp \rbra{ \| f \|_{L^2(\d s)}^2 } 
\label{} \\
\le& \varphi_{2V}(0) \exp \rbra{ \| f \|_{L^2(\d s)}^2 
+ \frac{2}{C_{2V}} 
\| f \|_{L^1(\d s)} } . 
\label{}
\end{align}
By Lemma \ref{lem3-2}, we see that 
\begin{align}
B 
=& \int_0^t \frac{\d u}{\sqrt{2 \pi u}} \e^{-2u} 
\rbra{ \Pi^{(u)} \bullet R } \sbra{ \absol{ \cE(f(\cdot + t);\theta_t X) -1 }^2 } 
\label{} \\
=& \int_0^t \frac{\d u}{\sqrt{2 \pi u}} \e^{-2u} 
R \sbra{ R_{X_{t-u}} \sbra{ \absol{ \cE(f(\cdot + t);\cdot) -1 }^2 } } 
\label{} \\
\le& E(t) \int_0^{\infty } \frac{\d u}{\sqrt{2 \pi u}} \e^{-2u} . 
\label{}
\end{align}
Therefore we see that \eqref{eq: lem3-1} is dominated by $ E(t) $ 
up to a multiplicative constant. 
The proof is now completed by (ii) of Lemma \ref{lem L1}. 
\end{proof}

\subsection{The third step}

In what follows, 
we take and utilize a non-negative, bounded, continuous function $ v_0 $ on $ \bR $ 
such that $ v_0(x) = 1 $ for $ |x| \le 2 $ and $ v_0(x)=0 $ for $ |x| \ge 3 $. 
We write $ v_1 = 1_{[-1,1]} $. 
We set $ V_0(\d x) = v_0(x) \d x $ and $ V_1(\d x) = v_1(x) \d x $. 
For any $ V $, we write 
\begin{align}
\Gamma(V;X) = \e^{-g(X)} \cK(V;X) . 
\label{}
\end{align}

\begin{Lem} \label{lem0}
Let $ h_t = \int_0^t f(s) \d s $ with $ f \in L^1(\d s) $. Suppose that 
\begin{align}
\int_T^{\infty } |f(s)| \d s \le 1 
\label{eq: T}
\end{align}
for some $ 0<T<\infty $. 
Then it holds that 
\begin{align}
\cK(V_0;X+h_{\cdot \wedge t}) 
\le 
\cK(V_1;X+h_{\cdot \wedge T}) 
, \quad t \ge T . 
\label{eq: lem0}
\end{align}
\end{Lem}

\begin{proof}
Note that we have $ |h_t-h_T| \le 1 $ for any $ t \ge T $. 
If $ s \ge T $ satisfies $ |X_s + h_T| \le 1 $, then we have 
$ |X_s + h_{s \wedge t}| \le 2 $. 
Hence we have 
\begin{align}
\int_0^{\infty } v_0(X_s + h_{s \wedge t}) \d s 
\ge \int_0^{\infty } v_1(X_s + h_{s \wedge T}) \d s 
, \quad t \ge T . 
\label{}
\end{align}
This completes the proof. 
\end{proof}

\begin{Lem} \label{lem1}
Let $ h_t = \int_0^t f(s) \d s $ with $ f \in L^2(\d s) \cap L^1(\d s) $. 
Let $ 0<r<\infty $ 
and let $ G_r(X) $ be a non-negative, bounded, continuous $ \cF_r $-measurable functional. 
Then it holds that 
\begin{align}
\sW \sbra{ G_r(X+h_{\cdot \wedge t}) \Gamma(V_0;X+h_{\cdot \wedge t}) } 
\tend{}{t \to \infty } 
\sW \sbra{ G_r(X+h) \Gamma(V_0;X+h) } . 
\label{eq: step 2}
\end{align}
\end{Lem}

\begin{proof}
Note that $ g(X+h_{\cdot \wedge t}) \to g(X+h) $ 
as $ t \to \infty $, because $ h_{\cdot \wedge t} \to h $ uniformly. 
By the continuity assumptions on $ G_r $ and $ v $, we have 
\begin{align}
G_r(X+h_{\cdot \wedge t}) \Gamma(V_0;X+h_{\cdot \wedge t}) 
\to 
G_r(X+h) \Gamma(V_0;X+h) 
\label{}
\end{align}
for $ \sW(\d X) $-almost every path $ X $. 
Since $ G_r(X) $ is bounded, 
it suffices to find $ Z \in L^1(\sW) $ such that 
$ \Gamma(V_0;X+h_{\cdot \wedge t}) \le Z(X) $, $ \sW $-a.e. 
for any large $ t $; 
in fact, we may obtain \eqref{eq: step 2} 
by the dominated convergence theorem. 

Since $ f \in L^1(\d s) $, 
we may take $ T>0 $ such that \eqref{eq: T} holds. 
By Lemma \ref{lem0}, we have \eqref{eq: lem0}, and hence we have 
\begin{align}
\Gamma(V_0;X+h_{\cdot \wedge t}) 
\le \cK(V_1;X+h_{\cdot \wedge T}) 
, \quad t \ge T . 
\label{}
\end{align}
Since $ M_T[\cK(V_1;\cdot);X] \in L^2(W) $ by \eqref{eq: LpW}, 
we see, by Proposition \ref{thm: main lem}, that 
\begin{align}
\cK(V_1;X+h_{\cdot \wedge T}) \in L^1(\sW) . 
\label{}
\end{align}
Therefore this functional $ \cK(V_1;X+h_{\cdot \wedge T}) $ is as desired. 
\end{proof}

\subsection{Proof of Theorem \ref{thm: main}}

We now proceed to prove Theorem \ref{thm: main}. 

\begin{proof}[Proof of Theorem \ref{thm: main}.]
Let $ h_t = \int_0^t f(s) \d s $ with $ f \in L^2(\d s) \cap L^1(\d s) $. 
Let $ 0<s<\infty $ 
and let $ G_s(X) $ be a non-negative, bounded, continuous 
$ \cF_s $-measurable functional. 
Let $ T > 0 $. 
Then, by Proposition \ref{thm: main lem}, we have 
\begin{align}
\sW \sbra{ G_s(X+h_{\cdot \wedge T}) \Gamma(V_0;X+h_{\cdot \wedge T}) } 
= \sW \sbra{ G_s(X) \Gamma(V_0;X) \cE_T(f;X) } . 
\label{eq: X + h(t)}
\end{align}
By Lemma \ref{lem1}, we have 
\begin{align}
\sW \sbra{ G_s(X+h_{\cdot \wedge T}) \Gamma(V_0;X+h_{\cdot \wedge T}) } 
\tend{}{T \to \infty } 
\sW \sbra{ G_s(X+h) \Gamma(V_0;X+h) } . 
\label{}
\end{align}
By Lemma \ref{lem3}, we have 
\begin{align}
\sW \sbra{ G_s(X) \Gamma(V_0;X) \cE_T(f;X) } 
\to 
\sW \sbra{ G_s(X) \Gamma(V_0;X) \cE(f;X) } 
\label{}
\end{align}
along some sequence $ T=t(n) \to \infty $. 
Thus, taking the limit as $ T=t(n) \to \infty $ 
in both sides of \eqref{eq: X + h(t)}, 
we obtain 
\begin{align}
\sW \sbra{ G_s(X+h) \Gamma(V_0;X+h) } 
= \sW \sbra{ G_s(X) \Gamma(V_0;X) \cE(f;X) } . 
\label{}
\end{align}
Hence we obtain 
\begin{align}
\sW \sbra{ G(X+h) \Gamma(V_0;X+h) } 
= \sW \sbra{ G(X) \Gamma(V_0;X) \cE(f;X) } 
\label{}
\end{align}
for any non-negative $ \cF_{\infty } $-measurable functional $ G(X) $. 
Replacing $ G(X) $ by $ F(X) \Gamma(V_0;X)^{-1} $, 
we obtain the desired conclusion. 
\end{proof}

{\bf Acknowledgements.} 
The author would like to thank Professors Marc Yor, Tadahisa Funaki, 
Shinzo Watanabe and Ichiro Shigekawa 
for their useful comments which helped improve this paper.

\bibliographystyle{plain}

\end{document}